\documentclass{article}
\usepackage[utf8]{inputenc}
\usepackage{amsthm}
\usepackage{amstext}
\usepackage{amsmath}
\usepackage{amssymb}
\usepackage[all]{xy}
\usepackage{color}
\usepackage{verbatim}
\usepackage{graphicx}
\usepackage{authblk}
\usepackage{appendix}
\usepackage[numbers]{natbib}
\usepackage{mathrsfs}
\usepackage{perpage}
 \MakePerPage{footnote}

\usepackage{hyperref}
\hypersetup{colorlinks=true, citecolor=blue, linkcolor=blue}

\usepackage{footnote}

\usepackage{enumitem}

\newtheorem{definition}{Def\text{}inition}[section]
\newtheorem{theorem}[definition]{Theorem}
\newtheorem{example}[definition]{Example}

\newtheorem{proposition}[definition]{Proposition}

\newtheorem{question}[definition]{Question}
\newtheorem{notation}[definition]{Notation}

\usepackage{tikz-cd}

\begin{document}

\title{ \bf \large Relationships between some selection principles and star selection principles}
\author{ \small JAVIER CASAS-DE LA ROSA}
\date{}
\maketitle

\begin{abstract} 
Motivated by the definition of classical star selection principles, some selection principles were defined in \cite{CRT}, as well as several questions about relationships between the notions defined there with the classical star selection principles were posed. In this paper, we answer the questions posed in \cite{CRT}. In addition, we show that some of the selection principles defined in \cite{CRT} are equivalent by taking collections of complements; also, some other results are provided involving collections of refinements.
\end{abstract}

\emph{Key words.} (strongly) star Menger, (strongly) star-Rothberger, star selection principles, refinements.

\emph{2020 AMS Subject Classification}. Primary 54D20; Secondary 54A25.



\section{Introduction and preliminaries} \label{intro}

One of the fields of mathematics having a plentiful history started in the early 20th century by Borel \cite{BO}, Menger \cite{MEN}, Hurewicz \cite{H}, Rothberger \cite{R} (among others) is the Selection Principles Theory. This theory caught again the attention of many mathematicians after the systematic research on selection principles made by Scheepers \cite{MS1} and currently, this theory has several applications to General Topology, Function spaces, Hyperspaces, etc. Nowadays, many others selection principles have been defined as well as many kind of collections to be considered in such those principles. In particular, one important research line, initiated by Ko\v{c}inac in \cite{K} and \cite{K2}, was arose from this theory by considering the star versions of several classical selection principles. The researching on these star versions have become named as star selection principles theory (see \cite{K_survey}).\\

Nowadays, there are many works on star selection principles that have obtained several interesting results. Recently, the study of star selection principles on the hyperspaces theory started in \cite{xiy} and continued in \cite{CRT2}, \cite{CRT} \cite{DRT} and \cite{CMR}. In particular, in \cite{CRT}, the authors defined some selection principles, which are motivated by the Menger-type and Rothberger-type classical star selection principles, to do some kind of  characterizations on hyperspaces; relationships between these selection principles with their respective star selection principles are also established. The main goal of this work is to answer the questions posed in \cite{CRT}. In addition, some others results involving collections consisting of refinements and complements are established. 

\subsection{Notation and terminology}

Let $X$ be a set. For a subset $U\subseteq X$, a family $\mathcal{U}$ of subsets of $X$ and a collection $\mathscr{U}$ of families of subsets of $X$, we write:
\begin{center}
\begin{tabular}{lcl}
 $U^c$ & = & $X\backslash U$;\\
 $\mathcal{U}^c$ & = & $\{U^c:U\in \mathcal{U}\}$;\\
 $\mathscr{U}^c$ & = & $\{\mathcal{U}^c:\mathcal{U}\in \mathscr{U}\}$.
\end{tabular} 
\end{center}
Also, we denote by $[X]^{<\omega}$ the collection of all finite subsets of $X$. For a subset $A$ of $X$ and a collection $\mathcal{U}$ of subsets of $X$, the star of $A$ with respect to $\mathcal{U}$, denoted by $St(A,\mathcal{U})$, is the set $\bigcup\{U\in\mathcal{U}:U\cap A\neq\emptyset\}$; for $A=\{x\}$ with $x\in X$, we write $St(x,\mathcal{U})$ instead of $St(\{x\},\mathcal{U})$. Throughout this paper, all spaces are assumed to be regular, unless a specific separation axiom is indicated. For notation and terminology, we refer to \cite{E}.\\

We recall some classical well-known selection principles and its star versions. Given an infinite set $X$, let $\mathscr{A}$ and $\mathscr{B}$ be collections of families of subsets of $X$. In \cite{MS1}, Scheepers introduced the following general forms of classical selection principles:\\
\newline
$\mathbf{S_1(\mathscr{A},\mathscr{B})}$: For any sequence $\{\mathcal{A}_n:n\in\omega\}$ of elements of $\mathscr{A}$ there is a sequence $\{B_n:n\in\omega\}$ such that for each $n\in\omega$, $B_n\in \mathcal{A}_n$ and $\{B_n:n\in\omega\}$ is an element of $\mathscr{B}$.\\
\newline
$\mathbf{S_{fin}(\mathscr{A},\mathscr{B})}$: For any sequence $\{\mathcal{A}_n:n\in\omega\}$ of elements of $\mathscr{A}$ there is a sequence $\{\mathcal{B}_n:n\in\omega\}$ such that for each $n\in\omega$, $\mathcal{B}_n$ is a finite subset of $\mathcal{A}_n$ and $\bigcup\{\mathcal{B}_n:n\in\omega\}$ is an element of $\mathscr{B}$.\\

Given a topological space $X$, we denote by $\mathscr{O}$ the collection of all open covers of $X$. Thus, $S_{fin}(\mathscr{O},\mathscr{O})$ defines the classical Menger covering property (see \cite{MEN}) and $S_1(\mathscr{O},\mathscr{O})$ defines the classical Rothberger covering property (see \cite{R}).

The following star selection principles were introduced by Ko\v{c}inac in \cite[Definition 1.1, Definition 1.2]{K}. Henceforth, $\mathcal{K}$ will denote a family of subsets of $X$:\\
\newline
$\mathbf{S^*_1(\mathscr{A},\mathscr{B})}$: For any sequence $\{\mathcal{A}_n:n\in\omega\}$ of elements of $\mathscr{A}$, there is a sequence $\{B_n:n\in\omega\}$ such that $B_n\in\mathcal{A}_n$, $n\in\omega$, and $\{St(B_n,\mathcal{A}_n):n\in\omega\}\in\mathscr{B}$.\\
\newline
$\mathbf{S^*_{fin}(\mathscr{A},\mathscr{B})}$: For any sequence $\{\mathcal{A}_n:n\in\omega\}$ of elements of $\mathscr{A}$, there is a sequence $\{\mathcal{B}_n:n\in\omega\}$ such that $\mathcal{B}_n$ is a finite subset of $\mathcal{A}_n$, $n\in\omega$, and $\bigcup_{n\in\omega}\{St(B,\mathcal{A}_n):B\in\mathcal{B}_n\}\in\mathscr{B}$.\\
\newline
$\mathbf{SS^*_{\mathcal{K}}(\mathscr{A},\mathscr{B})}$: For any sequence $\{\mathcal{A}_n:n\in\omega\}$ of elements of $\mathscr{A}$, there is a sequence $\{K_n:n\in\omega\}$ of elements of $\mathcal{K}$ such that $\{St(K_n,\mathcal{A}_n):n\in\omega\}\in\mathscr{B}$.\\

When $\mathcal{K}$ is the collection of all finite (resp. one-point) subsets of $X$, it is denoted by $\mathbf{SS^*_{fin}(\mathscr{A},\mathscr{B})}$ (resp. $\mathbf{SS^*_1(\mathscr{A},\mathscr{B})}$) instead of $\mathbf{SS^*_{\mathcal{K}}(\mathscr{A},\mathscr{B})}$. Thus, $S^*_{fin}(\mathscr{O},\mathscr{O})$ defines the star-Menger property (SM), $SS^*_{fin}(\mathscr{O},\mathscr{O})$ defines the strongly star-Menger property (SSM), $S^*_1(\mathscr{O},\mathscr{O})$ defines the star-Rothberger property (SR) and $SS^*_1(\mathscr{O},\mathscr{O})$ defines the strongly star-Rothberger property (SSR).

Next diagram shows the relationships among these properties (in the diagram $M$ and $R$ are used to denote the Menger property and the Rothberger property, respectively). We refer the reader to \cite{K_survey} to see the current state of knowledge about these relationships with others.

\begin{figure}[h!]
\[
\begin{tikzcd}
R  \arrow[r] \arrow[d] & M \arrow[d ]\\
SSR \arrow[r] \arrow[d]  & SSM \arrow[d]  \\
SR \arrow[r] & SM
\end{tikzcd}
\]
\label{classicstarVersionsFig}
\end{figure}

It is worth to mention that in the class of paracompact Hausdorff spaces the three Menger-type properties, $M$, $SSM$, $SM$ are equivalent and the same situation holds for the three Rothberger-type properties $R$, $SSR$ and $SR$ (see \cite[Theorem 2.8]{K}). Actually, these equivalences also holds in paraLindel\"of spaces (see \cite[Theorem 2.10]{CGS}). However, this fact is not true in general. In other words, none of the arrows in the previous diagram reverse. In particular, we mention the following interesting example and we refer the reader to {\cite[Example 3.7]{BCKM}} for detail.

\begin{example}\label{Example}
There exists a Tychonoff space which is star-Rothberger (thus, star-Menger) but is not strongly star-Menger (thus, neither strongly star-Rothberger).
\end{example}
\begin{proof}
Let $\kappa$ be an uncountable cardinal and $D(\kappa)=\{d_\alpha:\alpha<\kappa\}$ be a discrete space of cardinality $\kappa$. Let $\alpha(D(\kappa))=D(\kappa)\cup\{\infty\}$ be the one-point compactification of $D(\kappa)$. Let $$X=\left(\alpha(D(\kappa))\times[0, \kappa^+)\right)\cup \left(D(\kappa)\times\{\kappa^+\}\right)$$ be the subspace of the product space $\alpha(D(\kappa))\times[0,\kappa^+]$. Then $X$ is a Tychonoff star-Rothberger space (therefore, star-Menger) that is not strongly star-Menger space (therefore, neither strongly star-Rothberger).
\end{proof}

Now, we list the selection principles introduced in \cite{CRT}, which are motivated by the star versions of the Menger and Rothberger properties. We begin by mentioning the Rothberger-type selection principles.

\begin{definition}[\cite{CRT}]\label{Rothberger type}
Let $X$ be a set and $\mathscr{A}$, $\mathscr{B}$ be collections of families of subsets of $X$. We say that $X$ satisfies the principle:\\
\newline
$\mathbf{DS}_1(\mathscr{A},\mathscr{B})$ if for any sequence $\{\mathcal{D}_n:n\in\omega\}$ of elements in $\mathscr{A}$, there is a sequence $\{D_n:n\in\omega\}$ such that, for each $n\in\omega$, $D_n\in \mathcal{D}_n$ and $\bigcup\{\mathcal{F}_n:n\in\omega\}\in\mathscr{B}$, where $\mathcal{F}_n=\{\bigcap\mathcal{F}:\mathcal{F}\in[\mathcal{D}_n]^{<\omega} \text{ and } D_n^c\cap E^c\neq\emptyset \text{, for each }E\in\mathcal{F}\}$.\\
\newline
$\mathbf{CS}_1(\mathscr{A},\mathscr{B})$ if for any sequence $\{\mathcal{U}_n:n\in\omega\}$ of elements in $\mathscr{A}$, there is a sequence $\{U_n:n\in\omega\}$ such that, for each $n\in\omega$, $U_n\in \mathcal{U}_n$ and $\bigcup\{\mathcal{V}_n:n\in\omega\}\in\mathscr{B}$, where $\mathcal{V}_n=\{\bigcup\mathcal{V}:\mathcal{V}\in[\mathcal{U}_n]^{<\omega} \text{ and } U_n\cap V\neq\emptyset \text{, for each }V\in\mathcal{V}\}$.\\
\newline
$\mathbf{SDS}_1(\mathscr{A},\mathscr{B})$ if for any sequence $\{\mathcal{D}_n:n\in\omega\}$ of elements in $\mathscr{A}$, there is a sequence $\{x_n:n\in\omega\}$ of elements in $X$ such that $\bigcup\{\mathcal{F}_n:n\in\omega\}\in\mathscr{B}$, where $\mathcal{F}_n=\{\bigcap\mathcal{F}:\mathcal{F}\in[\mathcal{D}_n]^{<\omega} \text{ and } x_n\in E^c \text{, for each }E\in\mathcal{F}\}$.\\
\newline
$\mathbf{SCS}_1(\mathscr{A},\mathscr{B})$ if for any sequence $\{\mathcal{U}_n:n\in\omega\}$ of elements in $\mathscr{A}$, there is a sequence $\{x_n:n\in\omega\}$ of elements in $X$ such that $\bigcup\{\mathcal{V}_n:n\in\omega\}\in\mathscr{B}$, where $\mathcal{V}_n=\{\bigcup\mathcal{V}:\mathcal{V}\in[\mathcal{U}_n]^{<\omega} \text{ and } x_n\in V\text{, for each }V\in\mathcal{V}\}$.\\
\end{definition}

On the other hand, the following notions are the Menger-type selection principles which were also introduced in \cite{CRT}.

\begin{definition}[\cite{CRT}]\label{Menger type}
Let $X$ be a set and $\mathscr{A}$, $\mathscr{B}$ be collections of families of subsets of $X$. We say that $X$ satisfies the principle:\\
\newline
$\mathbf{DS}_{fin}(\mathscr{A},\mathscr{B})$ if for any sequence $\{\mathcal{D}_n:n\in\omega\}$ of elements in $\mathscr{A}$, there is a sequence $\{\mathcal{D}_n^f:n\in\omega\}$ such that, for each $n\in\omega$, $\mathcal{D}_n^f=\{D_n^1,\dots, D_n^{k_n}\}\subseteq \mathcal{D}_n$ and $\bigcup\{\mathcal{F}_n^i:n\in\omega \text{ and } i\in\{1,\dots, k_n\}\}\in\mathscr{B}$, where $\mathcal{F}_n^i=\{\bigcap\mathcal{F}:\mathcal{F}\in[\mathcal{D}_n]^{<\omega} \text{ and } (D_n^i)^c\cap E^c\neq\emptyset \text{, for each }E\in\mathcal{F}\}$, for any $n\in\omega$ and $i\in\{1,\ldots, k_n\}$.\\
\newline
$\mathbf{CS}_{fin}(\mathscr{A},\mathscr{B})$ if for any sequence $\{\mathcal{U}_n:n\in\omega\}$ of elements in $\mathscr{A}$, there is a sequence $\{\mathcal{U}_n^f:n\in\omega\}$ such that, for each $n\in\omega$, $\mathcal{U}_n^f=\{U_n^1,\dots, U_n^{k_n}\}\subseteq \mathcal{U}_n$ and $\bigcup\{\mathcal{V}_n^i:n\in\omega \text{ and }i\in\{1,\ldots, k_n\}\}\in\mathscr{B}$, where $\mathcal{V}_n^i=\{\bigcup\mathcal{V}:\mathcal{V}\in[\mathcal{U}_n]^{<\omega} \text{ and } U_n^i\cap V\neq\emptyset \text{, for each }V\in\mathcal{V}\}$, for any $n\in\omega$ and $i\in\{1,\ldots, k_n\}$.\\
\newline
$\mathbf{SDS}_{fin}(\mathscr{A},\mathscr{B})$ if for any sequence $\{\mathcal{D}_n:n\in\omega\}$ of elements in $\mathscr{A}$, there is a sequence $\{D_n^f:n\in\omega\}$ such that, for each $n\in\omega$, $D_n^f=\{x_n^1,\ldots, x_n^{k_n}\}\subseteq X$ and $\bigcup\{\mathcal{F}_n^i:n\in\omega \text{ and }i\in\{1,\dots, k_n\}\}\in\mathscr{B}$, where $\mathcal{F}_n^i=\{\bigcap\mathcal{F}:\mathcal{F}\in[\mathcal{D}_n]^{<\omega} \text{ and } x_n^i\in E^c \text{, for each }E\in\mathcal{F}\}$, for any $n\in\omega$ and $i\in\{1,\dots, k_n\}$.\\
\newline
$\mathbf{SCS}_{fin}(\mathscr{A},\mathscr{B})$ if for any sequence $\{\mathcal{U}_n:n\in\omega\}$ of elements in $\mathscr{A}$, there is a sequence $\{U_n^f:n\in\omega\}$ such that, for each $n\in\omega$, $U_n^f=\{x_n^1,\dots, x_n^{k_n}\}\subseteq X$ and $\bigcup\{\mathcal{V}_n^i:n\in\omega \text{ and }i\in\{1,\dots, k_n\}\}\in\mathscr{B}$, where $\mathcal{V}_n^i=\{\bigcup\mathcal{V}:\mathcal{V}\in[\mathcal{U}_n]^{<\omega} \text{ and } x_n^i\in V\text{, for each }V\in\mathcal{V}\}$, for any $n\in\omega$ and $i\in\{1,\dots, k_n\}$.\\
\end{definition}

\section{(Non) Equivalences of some selection principles}

Several questions that involve some selection principles from Definitions \ref{Rothberger type} and \ref{Menger type} with the Rothberger-type and Menger-type star selection principles were asked in \cite{CRT}. The main goal of this section is to answer these questions. We start by answering affirmatively the following question that basically ask about the equivalence of the Rothberger-type selection principle $\mathbf{CS}_1(\mathscr{O},\mathscr{O})$ and the star-Rothberger property.

\begin{question}[{\cite[Question 4.8 (2)]{CRT}}]\label{Question 4.8(2)}
Let $(X, \tau)$ be a topological space and let $\mathscr{O}$ be the collection of open covers of $X$. Does the principles $\mathbf{CS}_1(\mathscr{O},\mathscr{O})$ and $\mathbf{S}^*_1(\mathscr{O},\mathscr{O})$ are equivalent?
\end{question}

In \cite[Proposition 4.6]{CRT}, the authors showed one direction of Question \ref{Question 4.8(2)}:

\begin{proposition}[\cite{CRT}]\label{CS1impliesS*1}
Let $X$ be a topological space. If $X$ satisfies $\mathbf{CS}_1(\mathscr{O},\mathscr{O})$, then $X$ satisfies $\mathbf{S}^*_1(\mathscr{O},\mathscr{O})$.
\end{proposition}

The converse is also true.

\begin{proposition}\label{S*1impliesCS1}
Let $X$ be a topological space. If $X$ satisfies $\mathbf{S}^*_1(\mathscr{O},\mathscr{O})$, then $X$ satisfies $\mathbf{CS}_1(\mathscr{O},\mathscr{O})$.
\end{proposition}
\begin{proof}
Let $\{\mathcal{U}_n:n\in\omega\}$ be a sequence of open covers of $X$. By hypothesis, we have that there exists a sequence  $\{U_n:n\in\omega\}$ such that $U_n\in\mathcal{U}_n$ and the collection $\{St(U_n, \mathcal{U}_n):n\in\omega\}$ is an open cover of $X$. For each $n\in\omega$, let $\mathcal{W}_n=\{U\in\mathcal{U}_n:U\cap U_n\neq\emptyset\}$ and we define, for each $n\in\omega$, $\mathcal{V}_n=\{\bigcup\mathcal{V}:\mathcal{V}\in[\mathcal{W}_n]^{<\omega}\}$. It is clear that, for each $n\in\omega$, $\mathcal{V}_n=\{\bigcup\mathcal{V}:\mathcal{V}\in[\mathcal{U}_n]^{<\omega} \text{ and } U_n\cap V\neq\emptyset \text{, for each }V\in\mathcal{V}\}$.

\emph{Claim.} For each $n\in\omega$, $\bigcup\mathcal{V}_n=St(U_n, \mathcal{U}_n)$.

Fix $n\in\omega$. If $x\in \bigcup\mathcal{V}_n$, then there exists $\mathcal{V}_0\in[\mathcal{U}_n]^{<\omega}$ so that $U_n\cap V\neq\emptyset$, for each $V\in\mathcal{V}_0$, and $x\in \bigcup\mathcal{V}_0$. Then, there is $U_0\in\mathcal{V}_0$ such that $x\in U_0$. Therefore, $x\in St(U_n, \mathcal{U}_n)$. On the other hand, if $x\in St(U_n, \mathcal{U}_n)$, then there is $U_0\in\mathcal{U}_n$ such that $U_0\cap U_n\neq\emptyset$ and $x\in U_0$. Put $\mathcal{V}=\{U_0\}$. Then, $\mathcal{V}\in[\mathcal{U}_n]^{<\omega}$ and is so that for each $V\in\mathcal{V}$, $V\cap U_n\neq\emptyset$. Thus, we have that $x\in\bigcup\mathcal{V}$ and $\bigcup\mathcal{V}\in \mathcal{V}_n$. Hence, $x\in\bigcup\mathcal{V}_n$.

Finally, let us show that $\bigcup\{\mathcal{V}_n:n\in\omega\}$ is an open cover of $X$. Let $x\in X$. Since the collection $\{St(U_n, \mathcal{U}_n):n\in\omega\}$ is an open cover of $X$, there exists $n_0\in\omega$ such that $x\in St(U_{n_0}, \mathcal{U}_{n_0})$. By the Claim, we have that $x\in\bigcup\mathcal{V}_{n_0}$. Thus, there exists $\bigcup\mathcal{V}_0\in\mathcal{V}_{n_0}$ such that $x\in\bigcup\mathcal{V}_0$. Thus, $\bigcup\{\mathcal{V}_n:n\in\omega\}$ is an open cover of $X$. Hence, $X$ satisfies $\mathbf{CS}_1(\mathscr{O},\mathscr{O})$.
\end{proof}

From Propositions \ref{CS1impliesS*1} and \ref{S*1impliesCS1}, we obtain the following theorem.

\begin{theorem}\label{Rothberger equivalence}
Let $X$ be a topological space. Then $X$ is star-Rothberger if and only if $X$ satisfies the principle $\mathbf{CS}_1(\mathscr{O},\mathscr{O})$.
\end{theorem}

One more question, from \cite{CRT}, that involves the principle $\mathbf{CS}_1(\mathscr{O},\mathscr{O})$ is the following which ask about the relationship with the strongly star-Rothberger property.

\begin{question}[{\cite[Question 4.8 (1)]{CRT}}]\label{Question 4.8(1)}
Let $(X, \tau)$ be a topological space and let $\mathscr{O}$ be the collection of open covers of $X$. Does the principle $\mathbf{CS}_1(\mathscr{O},\mathscr{O})$ implies $\mathbf{SS}^*_1(\mathscr{O},\mathscr{O})$?
\end{question}

Observe that, by Theorem \ref{Rothberger equivalence}, Question \ref{Question 4.8(1)} asks if the star selection principle $\mathbf{S}^*_1(\mathscr{O},\mathscr{O})$ implies the star selection principle $\mathbf{SS}^*_1(\mathscr{O},\mathscr{O})$. In other words, Question \ref{Question 4.8(1)} asks if the star-Rothberger property implies the strongly star-Rothberger property. By Example \ref{Example}, we know that there is a star-Rothberger space that is not strongly star-Rothberger. Thus, Question \ref{Question 4.8(1)} has a negative answer.\\

Another question, also posed in \cite{CRT}, which is similar to the previous one but asked for arbitrary collections of families of subsets is the following.

\begin{question}[{\cite[Question 4.5]{CRT}}]\label{Question 4.5}
Let $X$ be a set and $\mathscr{A}$ and $\mathscr{B}$ collections of families of subsets of $X$. Does the principle $\mathbf{CS}_1(\mathscr{A},\mathscr{B})$ implies the principle $\mathbf{SS}_1^*(\mathscr{A},\mathscr{B})$ (or $\mathbf{S}_1^*(\mathscr{A},\mathscr{B})$)?
\end{question}

In \cite{CRT}, the authors showed that if one consider a topological space $X$ and take $\mathscr{A}$ and $\mathscr{B}$ to be the collection $\mathscr{O}$ of all open covers of $X$, then the principle $\mathbf{CS}_1(\mathscr{O},\mathscr{O})$ implies the principle $\mathbf{S}_1^*(\mathscr{O},\mathscr{O})$ (see Theorem 4.6 in \cite{CRT}). However, this question has a negative answer for the general case as the following example shows.\\

\begin{example}\label{exampleRothbergercase}
There exist a set $X$ and collections $\mathscr{A}$ and $\mathscr{B}$ of families of subsets of $X$ such that $X$ satisfies the principle $\mathbf{CS}_1(\mathscr{A},\mathscr{B})$ but neither $X$ satisfies the principle $\mathbf{SS}_1^*(\mathscr{A},\mathscr{B})$ nor $\mathbf{S}_1^*(\mathscr{A},\mathscr{B})$.
\end{example}

\begin{proof}
Let $X=\omega$ and we set $\mathscr{A} = \mathscr{B} = \mathcal{P}([X]^{<\omega}\setminus\{\emptyset\})$. Observe that if $\mathcal{W}$ is any finite subset of $[X]^{<\omega}\setminus\{\emptyset\}$, then $\bigcup\mathcal{W}\in [X]^{<\omega}\setminus\{\emptyset\}$. From this fact, it is clear that for every sequence $\{\mathcal{U}_n:n\in\omega\}$ of elements in $\mathscr{A}$, there exists a sequence $\{U_n:n\in\omega\}$ such that $U_n\in\mathcal{U}_n$ ($n\in\omega$) and $\bigcup\{\mathcal{V}_n:n\in\omega\}\in \mathscr{B}$, where $\mathcal{V}_n=\{\bigcup\mathcal{V}:\mathcal{V}\in[\mathcal{U}_n]^{<\omega} \text{ and } U_n\cap V\neq\emptyset \text{, for each }V\in\mathcal{V}\}$. In fact, something stronger holds:\\
For every sequence $\{\mathcal{U}_n:n\in\omega\}$ of elements in $\mathscr{A}$ and for every sequence $\{U_n:n\in\omega\}$ with $U_n\in\mathcal{U}_n$ ($n\in\omega$), $\bigcup\{\mathcal{V}_n:n\in\omega\}\in \mathscr{B}$, where $\mathcal{V}_n=\{\bigcup\mathcal{V}:\mathcal{V}\in[\mathcal{U}_n]^{<\omega} \text{ and } U_n\cap V\neq\emptyset \text{, for each }V\in\mathcal{V}\}$.

Therefore, $X$ satisfies $\mathbf{CS}_1(\mathscr{A},\mathscr{B})$.\\

Now, let $\mathcal{U}$ be the collection of all initial segments in $\omega$. In other words, let $\mathcal{U}=\{\{0,\ldots, k\}: k\in\omega\}$ and, for each $n\in \omega$, we put $\mathcal{U}_n=\mathcal{U}$. Thus, the sequence $\{\mathcal{U}_n:n\in\omega\}$ consists of elements in $\mathscr{A}$.\\

\emph{Claim 1.} For any sequence $\{U_n:n\in\omega\}$ with $U_n\in\mathcal{U}_n$ ($n\in\omega$), $\{St(U_n,\mathcal{U}_n):n\in\omega\}\notin\mathscr{B}$.

Let $\{U_n:n\in\omega\}$ be any sequence such that, for each $n\in\omega$, $U_n\in\mathcal{U}_n$. Observe that if $U$ is any element of $[X]^{<\omega}\setminus\{\emptyset\}$, then $St(U,\mathcal{U})=\omega$. From this fact, it is follows that, for each $n\in\omega$, $St(U_n,\mathcal{U}_n)=\omega$. Thus, for each $n\in\omega$, $St(U_n,\mathcal{U}_n)\notin [X]^{<\omega}\setminus\{\emptyset\}$, and therefore, $\{St(U_n,\mathcal{U}_n):n\in\omega\}\notin\mathscr{B}$.

Hence, $X$ does not satisfy $\mathbf{S}_1^*(\mathscr{A},\mathscr{B})$.\\

\emph{Claim 2.} For any sequence $\{x_n:n\in\omega\}$ of points in $X$, $\{St(x_n,\mathcal{U}_n):n\in\omega\}\notin\mathscr{B}$.

Let $\{x_n:n\in\omega\}$ be any sequence of points in $X$. Using same argument as in Claim 1, for each $n\in\omega$, $St(x_n,\mathcal{U}_n)=\omega$. Thus, $\{St(x_n,\mathcal{U}_n):n\in\omega\}\notin\mathscr{B}$.

Hence, $X$ does not satisfy $\mathbf{SS}_1^*(\mathscr{A},\mathscr{B})$.
\end{proof}

\vspace{.5cm}


The following analogous question about the equivalence of the Menger-type selection principle $\mathbf{CS}_{fin}(\mathscr{O},\mathscr{O})$ and the star-Menger property was also asked in \cite{CRT}. Here, we also give a positive answer to this question and it is obtained from Proposition \ref{S*finimpliesCSfin} below.

\begin{question}[{\cite[Question 4.14 (2)]{CRT}}]\label{Question 4.14(2)}
Let $(X, \tau)$ be a topological space and let $\mathscr{O}$ be the collection of open covers of $X$. Does the principles $\mathbf{CS}_{fin}(\mathscr{O},\mathscr{O})$ and $\mathbf{S}^*_{fin}(\mathscr{O},\mathscr{O})$ are equivalent?
\end{question}

In \cite[Proposition 4.12]{CRT}, the authors showed one direction of Question \ref{Question 4.14(2)}:

\begin{proposition}[\cite{CRT}]\label{CSfinimpliesS*fin}
Let $X$ be a topological space. If $X$ satisfies $\mathbf{CS}_{fin}(\mathscr{O},\mathscr{O})$, then $X$ satisfies $\mathbf{S}^*_{fin}(\mathscr{O},\mathscr{O})$.
\end{proposition}

The converse is also true.

\begin{proposition}\label{S*finimpliesCSfin}
Let $X$ be a topological space. If $X$ satisfies $\mathbf{S}^*_{fin}(\mathscr{O},\mathscr{O})$, then $X$ satisfies $\mathbf{CS}_{fin}(\mathscr{O},\mathscr{O})$.
\end{proposition}
\begin{proof}
Let $\{\mathcal{U}_n:n\in\omega\}$ be a sequence of open covers of $X$. Since $X$ satisfies $\mathbf{S}^*_{fin}(\mathscr{O},\mathscr{O})$, there are $\mathcal{V}_n\in[\mathcal{U}_n]^{<\omega}$ ($n\in\omega$) such that $\bigcup_{n\in\omega}\{St(V, \mathcal{U}_n):V\in\mathcal{V}_n\}$ is an open cover of $X$. For each $n\in\omega$, we put $\mathcal{V}_n=\{U_n^1,\ldots, U_n^{k_n}\}$. Now, for each $n\in\omega$ and each $i\in\{1,\dots, k_n\}$, let $\mathcal{W}_n^i=\{U\in\mathcal{U}_n:U\cap U_n^i\neq\emptyset\}$ and we define $\mathcal{V}_n^i=\{\bigcup\mathcal{V}:\mathcal{V}\in[\mathcal{W}_n^i]^{<\omega}\}$, for each $n\in\omega$ and each $i\in\{1,\dots, k_n\}$. Note that, for each $n\in\omega$ and each $i\in\{1,\dots, k_n\}$, $\mathcal{V}_n^i=\{\bigcup\mathcal{V}:\mathcal{V}\in[\mathcal{U}_n]^{<\omega} \text{ and } U_n^i\cap V\neq\emptyset \text{, for each }V\in\mathcal{V}\}$.

\emph{Claim.} For each $n\in\omega$ and each $i\in\{1,\dots, k_n\}$, $\bigcup\mathcal{V}_n^i=St(U_n^i, \mathcal{U}_n)$.

Fix $n\in\omega$ and $i\in\{1,\dots, k_n\}$. If $x\in \bigcup\mathcal{V}_n^i$, then there exists $\mathcal{V}_0\in[\mathcal{U}_n]^{<\omega}$ so that $U_n^i\cap V\neq\emptyset$, for each $V\in\mathcal{V}_0$, and $x\in \bigcup\mathcal{V}_0$. Then, there is $V_0\in\mathcal{V}_0\subseteq\mathcal{U}_n$ such that $x\in V_0$ and $U_n^i\cap V_0\neq\emptyset$. Thus, $x\in St(U_n^i, \mathcal{U}_n)$. Now, if $x\in St(U_n^i, \mathcal{U}_n)$, then there exists $U_0\in\mathcal{U}_n$ such that $U_0\cap U_n^i\neq\emptyset$ and $x\in U_0$. Let $\mathcal{V}=\{U_0\}\in[\mathcal{U}_n]^{<\omega}$ and it follows that $\bigcup\mathcal{V}\in \mathcal{V}_n^i$ with $x\in\bigcup\mathcal{V}$. Thus, $x\in\bigcup\mathcal{V}_n^i$.

Let us prove that $\bigcup\{\mathcal{V}_n^i:n\in\omega, i\in\{1,\dots, k_n\}\}$ is an open cover of $X$. Let $x\in X$. Since the collection $\bigcup_{n\in\omega}\{St(V, \mathcal{U}_n):V\in\mathcal{V}_n\}$ is an open cover of $X$, there exists $n_0\in\omega$ and $V_0\in\mathcal{V}_{n_0}$ such that $x\in St(V_0, \mathcal{U}_{n_0})$. Note that $V_0=U_{n_0}^{i_0}$ for some $i_0\in\{1,\dots, k_{n_0}\}$. Then, by the Claim, $x\in\bigcup\mathcal{V}_{n_0}^{i_0}$ and hence, $x\in\bigcup\mathcal{V}_0$ for some $\bigcup\mathcal{V}_0\in\mathcal{V}_{n_0}^{i_0}$. Thus, $\bigcup\{\mathcal{V}_n^i:n\in\omega, i\in\{1,\dots, k_n\}\}$ is an open cover of $X$. Hence, $X$ satisfies $\mathbf{CS}_{fin}(\mathscr{O},\mathscr{O})$.
\end{proof}

From Propositions \ref{CSfinimpliesS*fin} and \ref{S*finimpliesCSfin}, we obtain the following theorem.

\begin{theorem}\label{Menger equivalence}
Let $X$ be a topological space. Then $X$ is star-Menger if and only if $X$ satisfies the principle $\mathbf{CS}_{fin}(\mathscr{O},\mathscr{O})$.
\end{theorem}

Another question, posed in \cite{CRT}, that involves the principle $\mathbf{CS}_{fin}(\mathscr{O},\mathscr{O})$ is the following and it essentially asks about the relationship with the strongly star-Menger property.

\begin{question}[{\cite[Question 4.14 (1)]{CRT}}]\label{Question 4.14(1)}
Let $(X, \tau)$ be a topological space and let $\mathscr{O}$ be the collection of open covers of $X$. Does the principle $\mathbf{CS}_{fin}(\mathscr{O},\mathscr{O})$ implies $\mathbf{SS}^*_{fin}(\mathscr{O},\mathscr{O})$?
\end{question}

Observe that, by Theorem \ref{Menger equivalence}, Question \ref{Question 4.14(1)} asks if the star selection principle $\mathbf{S}^*_{fin}(\mathscr{O},\mathscr{O})$ implies the star selection principle $\mathbf{SS}^*_{fin}(\mathscr{O},\mathscr{O})$. In other words, Question \ref{Question 4.14(1)} asks if the star-Menger property implies the strongly star-Menger property. The space mentioned in Example \ref{Example} is an example of a star-Menger space that is not strongly star-Menger. Hence, Question \ref{Question 4.8(1)} has a negative answer.\\

Another question from \cite{CRT} that is similar to the above but posed for arbitrary collections of families is the following.

\begin{question}[{\cite[Question 4.11]{CRT}}]\label{Question 4.11}
Let $X$ be a set and $\mathscr{A}$ and $\mathscr{B}$ collections of families of subsets of $X$. Does the principle $\mathbf{CS}_{fin}(\mathscr{A},\mathscr{B})$ implies the principle $\mathbf{SS}_{fin}^*(\mathscr{A},\mathscr{B})$ (or $\mathbf{S}_{fin}^*(\mathscr{A},\mathscr{B})$)?
\end{question}

Again, in the same paper, the authors showed that if one consider a topological space $X$ and $\mathscr{A}$ and $\mathscr{B}$ being the collection $\mathscr{O}$ of all open covers of $X$, then the principle $\mathbf{CS}_{fin}(\mathscr{O},\mathscr{O})$ implies the principle $\mathbf{S}_{fin}^*(\mathscr{O},\mathscr{O})$ (see Theorem 4.12 in \cite{CRT}). However, this question has also a negative answer for the general case as the following example shows.\\

\begin{example}\label{exampleMengercase}
There exist a set $X$ and collections $\mathscr{A}$ and $\mathscr{B}$ of families of subsets of $X$ such that $X$ satisfies the principle $\mathbf{CS}_{fin}(\mathscr{A},\mathscr{B})$ but neither $X$ satisfies the principle $\mathbf{SS}_{fin}^*(\mathscr{A},\mathscr{B})$ nor $\mathbf{S}_{fin}^*(\mathscr{A},\mathscr{B})$.
\end{example}

\begin{proof}
Again, we consider $X=\omega$ and $\mathscr{A} = \mathscr{B} = \mathcal{P}([X]^{<\omega}\setminus\{\emptyset\})$. By using same observation pointed out in the first part of the proof of Example \ref{exampleRothbergercase}, we easily get that $X$ satisfies the following (stronger) property:\\
For every sequence $\{\mathcal{U}_n:n\in\omega\}$ of elements in $\mathscr{A}$ and for every sequence $\{\mathcal{U}_n^f:n\in\omega\}$ with $\mathcal{U}_n^f=\{U_n^1,\ldots, U_n^{k_n}\}\in[\mathcal{U}_n]^{<\omega}$ ($n\in\omega$), $\bigcup\{\mathcal{V}_n^i:n\in\omega \text{ and }  i=1,\dots,k_n \}\in \mathscr{B}$, where $\mathcal{V}_n^i=\{\bigcup\mathcal{V}:\mathcal{V}\in[\mathcal{U}_n]^{<\omega} \text{ and } U_n^i\cap V\neq\emptyset \text{, for each }V\in\mathcal{V}\}$ (for each $n\in\omega$ and each $i\in\{1,\dots, k_n\}$).

Therefore, $X$ satisfies $\mathbf{CS}_{fin}(\mathscr{A},\mathscr{B})$.\\

On the other hand, we consider $\mathcal{U}=\{\{0,\ldots, m\}: m\in\omega\}$. Define, for each $n\in \omega$, $\mathcal{U}_n=\mathcal{U}$. Then, the sequence $\{\mathcal{U}_n:n\in\omega\}$ consists of elements in $\mathscr{A}$.\\

\emph{Claim 1.} For any sequence $\{\mathcal{V}_n:n\in\omega\}$ with $\mathcal{V}_n\in[\mathcal{U}_n]^{<\omega}$ ($n\in\omega$), $\bigcup_{n\in\omega}\{St(V,\mathcal{U}_n):V\in\mathcal{V}_n\}\notin\mathscr{B}$.

To see this fact, it is enough to note that if $V$ is any element of $[X]^{<\omega}\setminus\{\emptyset\}$, then $St(V,\mathcal{U})=\omega$. Thus, if $\{\mathcal{V}_n:n\in\omega\}$ is any sequence such that, for each $n\in\omega$, $\mathcal{V}_n\in[\mathcal{U}_n]^{<\omega}$, then $St(V,\mathcal{U}_n)\notin[X]^{<\omega}$ for any $V\in\mathcal{V}_n$. Therefore, $\bigcup_{n\in\omega}\{St(V,\mathcal{U}_n):V\in\mathcal{V}_n\}\notin\mathscr{B}$.

Hence, $X$ does not satisfy $\mathbf{S}_{fin}^*(\mathscr{A},\mathscr{B})$.\\

\emph{Claim 2.} For any sequence $\{F_n:n\in\omega\}$ of finite subsets of $X$, $\{St(F_n,\mathcal{U}_n):n\in\omega\}\notin\mathscr{B}$.

Let $\{F_n:n\in\omega\}$ be any sequence of finite subsets of $X$. Note that if some $F_n$ is empty, then $St(F_n,\mathcal{U}_n)=\emptyset$ which is not an element in $[X]^{<\omega}\setminus\{\emptyset\}$. Thus, $\{St(F_n,\mathcal{U}_n):n\in\omega\}\notin\mathscr{B}$ in this case. Now, assume that, for each $n\in\omega$, $F_n\neq\emptyset$. Again, $St(F_n,\mathcal{U}_n)=\omega$ (for each $n\in\omega$), which is not an element in $[X]^{<\omega}\setminus\{\emptyset\}$. Thus, in this case we also have $\{St(F_n,\mathcal{U}_n):n\in\omega\}\notin\mathscr{B}$.

Hence, $X$ does not satisfy $\mathbf{SS}_{fin}^*(\mathscr{A},\mathscr{B})$.

\end{proof}

\section{Collections defined by refinements and complements}

The Question \ref{Question 4.5} ask for some relationship between principles $\mathbf{CS}_1(\mathscr{A},\mathscr{B})$ and $\mathbf{S}_1^*(\mathscr{A},\mathscr{B})$ for arbitrarily collections of families $\mathscr{A}$ and $\mathscr{B}$. Although the equivalence between those principles is true when we take $\mathscr{A}$ and $\mathscr{B}$ to be the collection $\mathscr{O}$ of all open covers of a given space $X$(see Theorem \ref{Rothberger equivalence}), the Example \ref{exampleRothbergercase} shows that in general it is not true. However, we can define some collections of families of subsets of a given set $X$ so that we get some implications between these principles.

Recall that given two collections $\mathcal{A}$ and $\mathcal{B}$ of sets, we say that $\mathcal{A}$ refines $\mathcal{B}$, denoted by $\mathcal{A} \prec \mathcal{B}$, if for every $A\in\mathcal{A}$ there exists $B\in\mathcal{B}$ such that $A\subseteq B$. By using the notion of refinement,  we introduce the following terminology that allow us to state implications between selection principles for the Rothberger case and Menger case, respectively.

\begin{notation}
Given a collection $\mathscr{B}$ of families of subsets of a set $X$, we denote by 
$$\mathscr{R}_\mathscr{B}^-=\{\mathcal{R}: \text{ there exists }\mathcal{B}\in\mathscr{B} \text{ such that } \mathcal{R}\prec\mathcal{B}\}$$
and
$$\mathscr{R}_\mathscr{B}^+=\{\mathcal{R}: \text{ there exists }\mathcal{B}\in\mathscr{B} \text{ such that } \mathcal{B}\prec\mathcal{R}\}.$$
\end{notation}

\begin{theorem}\label{S1 implies CS1}
Let $X$ be a set and $\mathscr{A}$ and $\mathscr{B}$ collections of families of subsets of $X$. If $X$ satisfies the principle $\mathbf{S}_1^*(\mathscr{A},\mathscr{B})$, then $X$ satisfies the principle $\mathbf{CS}_1(\mathscr{A},\mathscr{R}_\mathscr{B}^-)$.
\end{theorem}
\begin{proof}
Let $\{\mathcal{U}_n:n\in\omega\}$ be a sequence of elements in $\mathscr{A}$. We take a sequence $\{U_n:n\in\omega\}$ such that $U_n\in\mathcal{U}_n$ $(n\in\omega)$ and $\{St(U_n, \mathcal{U}_n):n\in\omega\}\in\mathscr{B}$. For all $n\in\omega$, let $\mathcal{V}_n=\{\bigcup\mathcal{V}:\mathcal{V}\in[\mathcal{U}_n]^{<\omega} \text{ and } U_n\cap V\neq\emptyset \text{ for all } V\in\mathcal{V}\}$. We claim that the collection $\bigcup_{n\in\omega}\mathcal{V}_n\in\mathscr{R}_\mathscr{B}^-$. For that end, it is enough to show that $\bigcup_{n\in\omega}\mathcal{V}_n\prec\{St(U_n, \mathcal{U}_n):n\in\omega\}$. Let $\bigcup\mathcal{V}_0\in\mathcal{V}_{n_0}$ for some $n_0\in\omega$. By definition of $\mathcal{V}_{n_0}$, we know that $U_{n_0}\cap V\neq\emptyset$ for each $V\in\mathcal{V}_0$ and $\mathcal{V}_0\in[\mathcal{U}_{n_0}]^{<\omega}$. From this fact, it follows that $\bigcup\mathcal{V}_0\subseteq St(U_{n_0}, \mathcal{U}_{n_0})$. Thus, $\bigcup_{n\in\omega}\mathcal{V}_n\prec\{St(U_n, \mathcal{U}_n):n\in\omega\}$ and hence, $\bigcup_{n\in\omega}\mathcal{V}_n\in\mathscr{R}_\mathscr{B}^-$. Therefore, $X$ satisfies $\mathbf{CS}_1(\mathscr{A},\mathscr{R}_\mathscr{B}^-)$.
\end{proof}

\begin{theorem}\label{CS1 implies S1}
Let $X$ be a set and $\mathscr{A}$ and $\mathscr{B}$ collections of families of subsets of $X$. If $X$ satisfies the principle $\mathbf{CS}_1(\mathscr{A},\mathscr{B})$, then $X$ satisfies the principle $\mathbf{S}_1^*(\mathscr{A},\mathscr{R}_\mathscr{B}^+)$.
\end{theorem}
\begin{proof}
Let $\{\mathcal{U}_n:n\in\omega\}$ be a sequence of elements in $\mathscr{A}$. Since $X$ satisfies $\mathbf{CS}_1(\mathscr{A},\mathscr{B})$, there exists a sequence $\{U_n:n\in\omega\}$ such that $U_n\in\mathcal{U}_n$ $(n\in\omega)$ and $\bigcup_{n\in\omega}\mathcal{V}_n\in\mathscr{B}$, where $\mathcal{V}_n=\{\bigcup\mathcal{V}:\mathcal{V}\in[\mathcal{U}_n]^{<\omega} \text{ and } U_n\cap V\neq\emptyset \text{ for all } V\in\mathcal{V}\}$ for all $n\in\omega$. Let us show that the collection $\{St(U_n, \mathcal{U}_n):n\in\omega\}\in\mathscr{R}_\mathscr{B}^+$. For this, it suffices to note that $\bigcup_{n\in\omega}\mathcal{V}_n\prec\{St(U_n, \mathcal{U}_n):n\in\omega\}$. But this fact is easily obtained as in the proof of Theorem \ref{S1 implies CS1}. Hence, $\{St(U_n, \mathcal{U}_n):n\in\omega\}\in\mathscr{R}_\mathscr{B}^+$. Therefore, $X$ satisfies $\mathbf{S}_1^*(\mathscr{A},\mathscr{R}_\mathscr{B}^+)$.
\end{proof}

With similar ideas as in the proof of Theorems \ref{S1 implies CS1} and \ref{CS1 implies S1}, one can prove the following results for the Menger-type selection principles.

\begin{theorem}\label{Sfin implies CSfin}
Let $X$ be a set and $\mathscr{A}$ and $\mathscr{B}$ collections of families of subsets of $X$. If $X$ satisfies the principle $\mathbf{S}_{fin}^*(\mathscr{A},\mathscr{B})$, then $X$ satisfies the principle $\mathbf{CS}_{fin}(\mathscr{A},\mathscr{R}_\mathscr{B}^-)$.
\end{theorem}

\begin{theorem}\label{CSfin implies Sfin}
Let $X$ be a set and $\mathscr{A}$ and $\mathscr{B}$ collections of families of subsets of $X$. If $X$ satisfies the principle $\mathbf{CS}_{fin}(\mathscr{A},\mathscr{B})$, then $X$ satisfies the principle $\mathbf{S}_{fin}^*(\mathscr{A},\mathscr{R}_\mathscr{B}^+)$.
\end{theorem}

By doing slight modifications to the arguments used in the proofs of Theorems \ref{S1 implies CS1} and \ref{CS1 implies S1}, one can also prove the strong case for the Rothberger-type and Menger-type selection principles, respectively.

\begin{theorem}\label{strong case}
Let $X$ be a set and $\mathscr{A}$ and $\mathscr{B}$ collections of families of subsets of $X$. Then
\begin{enumerate}
    \item If $X$ satisfies the principle $\mathbf{SS}_1^*(\mathscr{A},\mathscr{B})$, then $X$ satisfies the principle $\mathbf{SCS}_1(\mathscr{A},\mathscr{R}_\mathscr{B}^-)$;
    \item If $X$ satisfies the principle $\mathbf{SCS}_1(\mathscr{A},\mathscr{B})$, then $X$ satisfies the principle $\mathbf{SS}_1^*(\mathscr{A},\mathscr{R}_\mathscr{B}^+)$;
    \item If $X$ satisfies the principle $\mathbf{SS}_{fin}^*(\mathscr{A},\mathscr{B})$, then $X$ satisfies the principle $\mathbf{SCS}_{fin}(\mathscr{A},\mathscr{R}_\mathscr{B}^-)$;
    \item If $X$ satisfies the principle $\mathbf{SCS}_{fin}(\mathscr{A},\mathscr{B})$, then $X$ satisfies the principle $\mathbf{SS}_{fin}^*(\mathscr{A},\mathscr{R}_\mathscr{B}^+)$.
\end{enumerate}
\end{theorem}

Now, we show some equivalences among selection principles, defined in \cite{CRT}, by considering the collection obtained by taking complements.

\begin{theorem}\label{CS1 equivalent DS1}
Let $X$ be a set and $\mathscr{A}$ and $\mathscr{B}$ collections of families of subsets of $X$. Then $X$ satisfies the principle $\mathbf{CS}_1(\mathscr{A},\mathscr{B})$ if and only if $X$ satisfies the principle $\mathbf{DS}_1(\mathscr{A}^c,\mathscr{B}^c)$.
\end{theorem}
\begin{proof}
$[\rightarrow]$ Assume $X$ satisfies $\mathbf{CS}_1(\mathscr{A},\mathscr{B})$ and let $\{\mathcal{D}_n:n\in\omega\}$ be a sequence of elements in $\mathscr{A}^c$. We put, for each $n\in\omega$, $\mathcal{U}_n=\mathcal{D}_n^c$. Then, the sequence $\{\mathcal{U}_n:n\in\omega\}$ consists of elements in $\mathscr{A}$. Thus, there exists a sequence $\{U_n:n\in\omega\}$ such that $U_n\in\mathcal{U}_n$ $(n\in\omega)$ and $\bigcup_{n\in\omega}\mathcal{V}_n\in\mathscr{B}$ where $\mathcal{V}_n=\{\bigcup\mathcal{V}:\mathcal{V}\in[\mathcal{U}_n]^{<\omega} \text{ and } U_n\cap V\neq\emptyset \text{ for all } V\in\mathcal{V}\}$ for all $n\in\omega$. For each $n\in\omega$, let $\mathcal{F}_n=\{(\bigcup\mathcal{V})^c:\mathcal{V}\in[\mathcal{U}_n]^{<\omega} \text{ and } U_n\cap V\neq\emptyset \text{ for all } V\in\mathcal{V}\}$. Note that we can rewrite these collections as $\mathcal{F}_n=\{\bigcap\mathcal{F}:\mathcal{F}\in[\mathcal{D}_n]^{<\omega} \text{ and } D_n^c\cap E^c\neq\emptyset \text{ for all } E\in\mathcal{F}\}$ where $D_n=U_n^c$ for all $n\in\omega$. Thus, each $D_n\in\mathcal{D}_n$. Furthermore, $\bigcup_{n\in\omega}\mathcal{F}_n\in\mathscr{B}^c$. Indeed, observe that $\bigcup_{n\in\omega}\mathcal{F}_n=\bigcup_{n\in\omega}(\mathcal{V}_n)^c=(\bigcup_{n\in\omega}\mathcal{V}_n)^c$ and, since $\bigcup_{n\in\omega}\mathcal{V}_n\in\mathscr{B}$, it follows that  $\bigcup_{n\in\omega}\mathcal{F}_n\in\mathscr{B}^c$.

Hence, $X$ satisfies $\mathbf{DS}_1(\mathscr{A}^c,\mathscr{B}^c)$.\\


$[\leftarrow]$ Assume $X$ satisfies $\mathbf{DS}_1(\mathscr{A}^c,\mathscr{B}^c)$ and let $\{\mathcal{U}_n:n\in\omega\}$ be a sequence of elements in $\mathscr{A}$. For each $n\in\omega$, let $\mathcal{D}_n=\mathcal{U}_n^c$. Then, the sequence $\{\mathcal{D}_n:n\in\omega\}$ consists of elements in $\mathscr{A}^c$. By applying the hypothesis, we get a sequence $\{D_n:n\in\omega\}$ such that $D_n\in\mathcal{D}_n$ $(n\in\omega)$ and $\bigcup_{n\in\omega}\mathcal{F}_n\in\mathscr{B}^c$ where $\mathcal{F}_n=\{\bigcap\mathcal{F}:\mathcal{F}\in[\mathcal{D}_n]^{<\omega} \text{ and } D_n^c\cap E^c\neq\emptyset \text{ for all } E\in\mathcal{F}\}$ for all $n\in\omega$. We define, for each $n\in\omega$, $U_n=D_n^c$. Since $D_n^c\in\mathcal{D}_n^c$ and $\mathcal{D}_n^c=\mathcal{U}_n$, then each $U_n\in\mathcal{U}_n$ ($n\in\omega$). In addition, we define for each $n\in\omega$, $\mathcal{V}_n=\{(\bigcap\mathcal{F})^c:\mathcal{F}\in[\mathcal{D}_n]^{<\omega} \text{ and } D_n^c\cap E^c\neq\emptyset \text{ for all } E\in\mathcal{F}\}$. Thus, we can rewrite these collections as $\mathcal{V}_n=\{\bigcup\mathcal{V}:\mathcal{V}\in[\mathcal{U}_n]^{<\omega} \text{ and } U_n\cap V\neq\emptyset \text{ for all } V\in\mathcal{V}\}$. It only remains to show that $\bigcup_{n\in\omega}\mathcal{V}_n\in\mathscr{B}$.
For that end, observe that each $\mathcal{V}_n=\mathcal{F}_n^c$ and then, $\bigcup_{n\in\omega}\mathcal{V}_n=\bigcup_{n\in\omega}(\mathcal{F}_n)^c=(\bigcup_{n\in\omega}\mathcal{F}_n)^c$ and, since $\bigcup_{n\in\omega}\mathcal{F}_n\in\mathscr{B}^c$, it follows that  $\bigcup_{n\in\omega}\mathcal{V}_n\in\mathscr{B}$.

Therefore, $X$ satisfies $\mathbf{CS}_1(\mathscr{A},\mathscr{B})$.
\end{proof}

Now, let us show the equivalence for strong version of Theorem \ref{CS1 equivalent DS1}.

\begin{theorem}\label{SCS1 equivalent SDS1}
Let $X$ be a set and $\mathscr{A}$ and $\mathscr{B}$ collections of families of subsets of $X$. Then $X$ satisfies the principle $\mathbf{SCS}_1(\mathscr{A},\mathscr{B})$ if and only if $X$ satisfies the principle $\mathbf{SDS}_1(\mathscr{A}^c,\mathscr{B}^c)$.
\end{theorem}
\begin{proof}
$[\rightarrow]$ Assume $X$ satisfies $\mathbf{SCS}_1(\mathscr{A},\mathscr{B})$ and let $\{\mathcal{D}_n:n\in\omega\}$ be a sequence of elements in $\mathscr{A}^c$. We put, for each $n\in\omega$, $\mathcal{U}_n=\mathcal{D}_n^c$. Then, the sequence $\{\mathcal{U}_n:n\in\omega\}\subseteq \mathscr{A}$. Thus, there exists a sequence $\{x_n:n\in\omega\}$ of points in $X$ such that $\bigcup_{n\in\omega}\mathcal{V}_n\in\mathscr{B}$ where $\mathcal{V}_n=\{\bigcup\mathcal{V}:\mathcal{V}\in[\mathcal{U}_n]^{<\omega} \text{ and } x_n\in V \text{ for all } V\in\mathcal{V}\}$ for all $n\in\omega$. For each $n\in\omega$, let $\mathcal{F}_n= \mathcal{V}_n^c$. It means that, for each $n\in\omega$, $\mathcal{F}_n=\{(\bigcup\mathcal{V})^c:\mathcal{V}\in[\mathcal{U}_n]^{<\omega} \text{ and } x_n\in V \text{ for all } V\in\mathcal{V}\}$. Rewriting these collections as $\mathcal{F}_n=\{\bigcap\mathcal{F}:\mathcal{F}\in[\mathcal{D}_n]^{<\omega} \text{ and } x_n\in E^c \text{ for all } E\in\mathcal{F}\}$. In addition, $\bigcup_{n\in\omega}\mathcal{F}_n\in\mathscr{B}^c$, since  $\bigcup_{n\in\omega}\mathcal{F}_n=\bigcup_{n\in\omega}\mathcal{V}_n^c=(\bigcup_{n\in\omega}\mathcal{V}_n)^c$ and $\bigcup_{n\in\omega}\mathcal{V}_n\in\mathscr{B}$.

Hence, $X$ satisfies $\mathbf{SDS}_1(\mathscr{A}^c,\mathscr{B}^c)$.\\


$[\leftarrow]$ Assume $X$ satisfies $\mathbf{SDS}_1(\mathscr{A}^c,\mathscr{B}^c)$ and let $\{\mathcal{U}_n:n\in\omega\}\subseteq\mathscr{A}$. For each $n\in\omega$, we put $\mathcal{D}_n=\mathcal{U}_n^c$. Then, $\{\mathcal{D}_n:n\in\omega\}\subseteq\mathscr{A}^c$. Thus, there exists a sequence $\{x_n:n\in\omega\}$ of points of $X$ such that $\bigcup_{n\in\omega}\mathcal{F}_n\in\mathscr{B}^c$ where $\mathcal{F}_n=\{\bigcap\mathcal{F}:\mathcal{F}\in[\mathcal{D}_n]^{<\omega} \text{ and } x_n\in E^c \text{ for all } E\in\mathcal{F}\}$ for all $n\in\omega$. For each $n\in\omega$, we put $\mathcal{V}_n= \mathcal{F}_n^c$. It means that, for each $n\in\omega$, $\mathcal{V}_n=\{(\bigcap\mathcal{F})^c:\mathcal{F}\in[\mathcal{D}_n]^{<\omega} \text{ and } x_n\in E^c \text{ for all } E\in\mathcal{F}\}$. Moreover, we can rewrite these collections as $\mathcal{V}_n=\{\bigcup\mathcal{V}:\mathcal{V}\in[\mathcal{U}_n]^{<\omega} \text{ and } x_n\in V \text{ for all } V\in\mathcal{V}\}$. Finally, observe that $\bigcup_{n\in\omega}\mathcal{V}_n=\bigcup_{n\in\omega}\mathcal{F}_n^c=(\bigcup_{n\in\omega}\mathcal{F}_n)^c$ with $\bigcup_{n\in\omega}\mathcal{F}_n\in\mathscr{B}^c$. Hence, $\bigcup_{n\in\omega}\mathcal{V}_n\in\mathscr{B}$. 

Therefore, $X$ satisfies $\mathbf{SCS}_1(\mathscr{A},\mathscr{B})$.
\end{proof}

With similar ideas as in Theorems \ref{CS1 equivalent DS1} and \ref{SCS1 equivalent SDS1} one can proves the respective equivalences for the Menger case.

\begin{theorem}\label{CSfin equivalent DSfin}
Let $X$ be a set and $\mathscr{A}$ and $\mathscr{B}$ collections of families of subsets of $X$. Then $X$ satisfies the principle $\mathbf{CS}_{fin}(\mathscr{A},\mathscr{B})$ if and only if $X$ satisfies the principle $\mathbf{DS}_{fin}(\mathscr{A}^c,\mathscr{B}^c)$.
\end{theorem}

\begin{theorem}\label{SCSfin equivalent SDSfin}
Let $X$ be a set and $\mathscr{A}$ and $\mathscr{B}$ collections of families of subsets of $X$. Then $X$ satisfies the principle $\mathbf{SCS}_{fin}(\mathscr{A},\mathscr{B})$ if and only if $X$ satisfies the principle $\mathbf{SDS}_{fin}(\mathscr{A}^c,\mathscr{B}^c)$.
\end{theorem}

\section*{Acknowledgements}

The author was supported for this research by the Consejo Nacional de Ciencia y Tecnolog\'ia CONACYT, M\'exico, Scholarship 769010.

\small
\baselineskip=5pt


\textsc{Department of Mathematics and Statistics, York University, 4700 Keele St. Toronto, ON M3J 1P3 Canada}\par\nopagebreak

\vspace{.1cm}
\textit{Email address}: J. Casas-de la Rosa: \texttt{olimpico.25@hotmail.com}


\begin{thebibliography}{99}

\bibitem{BCKM} M. Bonanzinga, F. Cammaroto, Lj.D.R. Ko\v{c}inac, M.V. Matveev, On weaker forms of Menger, Rothberger and Hurewicz properties, \textit{Mat. Vesnik}, 61, (2009), {13}--{23}.

\bibitem{BO} E. Borel, \emph{Sur la classification des ensembles de mesure nulle}, Bull. Soc. Math. de France. 47 (1919) 97-125.

\bibitem{CGS} J. Casas-de la Rosa, S. A. Garcia-Balan, P. J. Szeptycki, Some star and strongly star selection principles, \textit{Topology Appl}. 258 (2019), 572-587. \href{https://doi.org/10.1016/j.topol.2017.11.034}{https://doi.org/10.1016/j.topol.2017.11.034}

\bibitem{CMR} J. Casas-de la Rosa, I. Mart\'inez-Ruiz, A. Ram\'irez-P\'aramo, \emph{Star versions of the Menger property on hyperspaces}, Houston J. Math., to appear.

\bibitem{xiy} J. Casas-de la Rosa, I. Mart\'inez-Ruiz, A. Ram\'irez-P\'aramo, \emph{Star versions of the Rothberger property on hyperspaces}, Topol. Appl. 283 (2020), Art. ID 107396, 12 pages.

\bibitem{CRT2} R. Cruz-Castillo, A. Ram\'irez-P\'aramo, J.F. Tenorio \emph{Menger and Menger-type star selection principles for hit-and-miss topology}, Topol. Appl. 290 (2021), Art. ID 107574, 12 pages.

\bibitem{CRT} R. Cruz-Castillo, A. Ram\'irez-P\'aramo, J.F. Tenorio \emph{Star and strong star-type versions of Rothberger and Menger principles for hit-and-miss topology}, Topol. Appl. 300 (2021), Art. ID 107758, 11 pages.

\bibitem{DRT} J. D\'iaz-Reyes, A. Ram\'irez-P\'aramo, J.F. Tenorio \emph{Rothberger and Rothberger-type star selection principles on hyperspaces}, Topol. Appl. 287 (2021), Art. ID 107448, 9 pages.

\bibitem{E} R. Engelking, General Topology, \textit{Heldermann Verlag, Berlin, Sigma Series in Pure Mathematics} 6, (1989).

\bibitem{H} W. Hurewicz, \"{U}ber eine Verallgemeinerung des Borelschen Theorems, \textit{Math. Z}., 24 (1) (1925), {401}--{421}. \href{https://doi.org/10.1007/BF01216792}{https://doi.org/10.1007/BF01216792}

\bibitem{K} Lj.D.R. Ko\v{c}inac, Star-Menger and related spaces, \textit{Publ. Math. Debrecen}, 55 (1999), {421}--{431}.

\bibitem{K2} Lj.D.R. Ko\v{c}inac, Star-Menger and related spaces II, \textit{Filomat} 13 (1999), {129}--{140}.

\bibitem{K_survey} Lj.D.R. Ko\v{c}inac, Star selection principles: A survey, \textit{Khayyam J. Math.} 1 (2015), 1, {82}--{106}. \href{10.22034/KJM.2015.12289}{10.22034/KJM.2015.12289}

\bibitem{MEN} K. Menger, Einige \"Uberdeckungss\"atze der Punltmengenlehre, \textit{Sitzungsberichte Abt. 2a, Mathematik, Astronomie, Physik, Meteorologie und Mechanik} (Wiener Akademie, Wien) 133 (1924), {421}--{444}.

\bibitem{R} F. Rothberger, Eine Versch\"arfung der Eigenschaft C, \textit{Fund. Math}., 30 (1938), {50}--{55}. \href{https://www.impan.pl/pl/wydawnictwa/czasopisma-i-serie-wydawnicze/fundamenta-mathematicae/all/30/0/111967/eine-verscharfung-der-eigenschaft-c}{10.4064/fm-30-1-50-55}

\bibitem{MS1} M. Scheepers, Combinatorics of open covers I: Ramsey Theory, \textit{Topol. Appl}., 69 (1996), {31}--{62}. \href{https://doi.org/10.1016/0166-8641(95)00067-4}{https://doi.org/10.1016/0166-8641(95)00067-4}

\vspace{1cm}

\end{thebibliography}
\end{document}